\documentclass{bmcart}

\usepackage[utf8]{inputenc} 
\usepackage{tikz}
\usepackage{amsfonts}
\usepackage{amssymb}
\usepackage{amsmath}
\usepackage{amsthm}
\usepackage[mathscr]{euscript}
\usepackage{thmtools}
\usepackage{enumerate}
\usepackage{color}
\usepackage[colorlinks]{hyperref} 
\hypersetup{colorlinks=true,citecolor=blue, linkcolor = blue, urlcolor = black}
\usepackage{setspace}

\def\includegraphics{}

\startlocaldefs
\newtheorem{theorem}[subsection]{Theorem}
\newtheorem{coro}[subsection]{Corollary}

\theoremstyle{remark}
\newtheorem*{remark}{Remark}
\newtheorem*{defn}{Definition}

\newcommand{\rra}{\rightarrow}
\newcommand{\inv}{^{-1}}
\newcommand{\brk}[1]{ \left\lbrace #1 \right\rbrace }
\newcommand{\pwr}[1]{\left( #1 \right)}

\DeclareMathOperator{\Real}{Re}

\endlocaldefs

\begin{document}

\begin{frontmatter}

\begin{fmbox}
\dochead{Research}


\title{Partition zeta functions}


\author[
   addressref={aff1},                   
   email={robert.schneider@emory.edu}   
]{\inits{RS}\fnm{Robert} \snm{Schneider}}


\address[id=aff1]{
  \orgname{Department of Mathematics and Computer Science, Emory University}, 
  \city{Atlanta, Georgia 30322},                              
  \cny{USA}                                    
}



\end{fmbox}


\begin{abstractbox}

\begin{abstract} 
We exploit transformations relating generalized $q$-series, infinite products, sums over integer partitions, and continued fractions, to find partition-theoretic formulas to compute the values of constants such as $\pi$, and to connect sums over partitions to the Riemann zeta function, multiple zeta values, and other number-theoretic objects.
\end{abstract}




\end{abstractbox}
%

\end{frontmatter}



\section{Introduction, notations and central theorem}
One marvels at the degree to which our contemporary understanding of $q$-series, integer partitions, and what is now known as the Riemann zeta function $\zeta (s)$ emerged nearly fully-formed from Euler's pioneering work \cite{andrews1998theory, dunham1999euler}. Euler discovered the magical-seeming generating function for the partition function $p(n)$
\begin{equation}\label{1}
\frac{1}{(q;q)_{\infty}} = \sum_{n=0}^{\infty}p(n)q^n,
\end{equation}
in which the $q$-Pochhammer symbol is defined as $(z;q)_0:=1, (z;q)_n := \prod_{k=0}^{n-1}(1-zq^k)$ for $n\geq 1$, and $(z;q)_{\infty} = \lim_{n\rra \infty}(z;q)_n$ if the product converges, where we take $z \in \mathbb C$ and $q := e^{i2\pi \tau}$ with $\tau \in \mathbb H$ (the upper half-plane). He also discovered the beautiful product formula relating the zeta function $\zeta(s)$ to the set $\mathbb P$ of primes
\begin{equation}\label{2}
\frac{1}{\prod_{p\in \mathbb P}\pwr{1 - \frac{1}{p^s}}} = \sum_{n=1}^{\infty} \frac{1}{n^s} := \zeta(s), \hspace{.3em} \Real(s) >1.
\end{equation}
In this paper, we see \eqref{1} and \eqref{2} are special cases of a single partition-theoretic formula. Euler used another product identity for the sine function
\begin{equation}\label{3}
x\prod_{n=1}^{\infty} \pwr{1 - \frac{x^2}{\pi^2n^2}} = \sin x
\end{equation}
to solve the so-called Basel problem, finding the exact value of $\zeta (2)$; he went on to find an exact formula for $\zeta (2k)$ for every $k \in \mathbb Z^+$ \cite{dunham1999euler}. Euler's approach to these problems, interweaving infinite products, infinite sums and special functions, permeates number theory.

Very much in the spirit of Euler, here we consider certain series of the form $\sum_{\lambda \in \mathcal P}\phi(\lambda)$, where the sum is taken over the set $\mathcal P$ of integer partitions $\lambda = (\lambda_1,\lambda_2,\dots, \lambda_r), \lambda_1 \geq \lambda_2 \geq \cdots \geq \lambda_r \geq 1$, as well as the ``empty partition" $\emptyset$, and where $\phi : \mathcal P \rra \mathbb C$. We might sum instead over a subset of $\mathcal P$, and will intend $\mathcal P_X \subseteq \mathcal P$ to mean the set of partitions whose parts all lie in $X \subseteq \mathbb Z^+$. 

A few other notations should be fixed and comments made. We call the number of parts of $\lambda$ the \textit{length} $\ell(\lambda) := r$ of the partition. We call the number being partitioned the \textit{size} $|\lambda|:= \lambda_1 + \lambda_2 + \cdots + \lambda_r$ of $\lambda$ (also referred to as the \textit{weight} of the partition). We write $\lambda \vdash n$ to indicate $\lambda$ is a partition of $n$ (i.e., $|\lambda| = n$), and we allow a slight abuse of notation to let ``$\lambda_i \in \lambda$" indicate $\lambda_i$ is one of the parts of $\lambda$ (with multiplicity). Furthermore, for formal transparency, we define the natural number $n_{\lambda}$, which we call the \textit{integer} of $\lambda$, to be the product of its parts, i.e.,
$$n_{\lambda}:=\lambda_1\lambda_2\cdots \lambda_r.$$ We assume the conventions $\ell(\emptyset) := 0, |\emptyset|:= 0$, and $n_{\emptyset} := 1$ (being an empty product), and take $s \in \mathbb C$, $\Real (s) >1$ and $|q| < 1$ throughout, unless otherwise specified. Proofs are postponed until Section \ref{Proofs of theorems and corollaries}.

Sums of the form $\sum_{\lambda}\phi(\lambda)$ obey many interesting transformations, and often reveal patterns that are otherwise obscure, much as with sums over natural numbers. MacMahon appears to be the first to have considered such summations explicitly, looking at sums over partitions of a given positive integer in \cite{macmahon1984combinatory}. Fine gives a variety of beautiful results and techniques related to sums over partitions in \cite{fine1988basic}, as do Andrews \cite{andrews1998theory}, Alladi \cite{alladi1997partition}, and other authors. More recent work of Bloch-Okounkov \cite{bloch2000character} and Zagier \cite{zagier2015partitions} relates sums over partitions to infinite families of quasimodular forms via the $q$-bracket operator, and Griffin-Ono-Warnaar \cite{griffin2014framework} use partition sums involving Hall-Littlewood polynomials to produce modular functions. These series have deep connections. It is natural then to wonder, in what other ways might sums over partitions connect to classical number-theoretic objects?

We need to introduce one more notation, in order to state the central theorem. Define $\varphi_n(f;q)$ by $\varphi_0(f;q) := 1$ and 
\begin{equation*}
\varphi_n(f;q) := \prod_{k=1}^n (1-f(k)q^k)
\end{equation*}
where $n\geq 1$, for an arbitrary function $f : \mathbb N \rra \mathbb C$. When the infinite product converges, let $\varphi_{\infty}(f;q):= \lim_{n\rra \infty}\varphi_n(f;q)$. We think of $\varphi$ as a generalization of the $q$-Pochhammer symbol. Note that if we set $f$ equal to a constant $z$, then $\varphi$ does specialize to the $q$-Pochhammer symbol, as $\varphi_n(z;q) = (zq;q)_n $ and $\varphi_{\infty}(z;q) = (zq;q)_{\infty}$. 

As in \eqref{1} and \eqref{2}, it is the reciprocal $1/\varphi_{\infty}(f;q)$ that interests us. With the above notations, we have the following system of identities.

\begin{theorem}\label{1.1}
If the product converges, then $1/\varphi_{\infty}(f;q) = \prod_{n=1}^{\infty}(1-f(n)q^n)^{-1}$ may be expressed in a number of equivalent forms, viz.
\begin{align}
\frac{1}{\varphi_{\infty}(f;q)} & =  \sum_{\lambda \in \mathcal P} q^{|\lambda|} \prod_{\lambda_i \in \lambda} f(\lambda_i) \label{4}\\ \displaybreak
& =  1+ \sum_{n=1}^{\infty} q^n \frac{f(n)}{\varphi_n(f;q)} \label{5}\\ 
& =  1+ \frac{1}{\varphi_{\infty}(f;q)} \sum_{n=1}^{\infty} q^n f(n)\varphi_{n-1}(f;q) \label{6} \\ 
& =  1 + \sum_{n=1}^{\infty} \frac{(-1)^n (q\inv)^{\frac{n(n-1)}{2}}}{\varphi_n\pwr{\frac{1}{f};q\inv} \prod_{k=1}^{n-1}f(k)} \label{7}  \\
& = 1+\cfrac{\sum_{(6)}}{1-\cfrac{\sum_{(5)}}{1+\cfrac{\sum_{(6)}}{1-\cfrac{\sum_{(5)}}{1+ \cdots}}}} \label{8}
\end{align}
where $\sum_{(5)},\sum_{(6)}$ in \eqref{8} denote the summations appearing in \eqref{5} and \eqref{6}, respectively.
\end{theorem}
The product on the right-hand side of identity \eqref{4} above is taken over the parts $\lambda_i$ of $\lambda$. Note that the summation in \eqref{7} converges for $q\inv$ outside the unit circle (it may converge inside the circle as well). Note also that, by L`Hospital's rule, any power series $\sum_{n=1}^{\infty}f(n)q^n$ with constant term zero can be written as the limit
\begin{equation*}
\sum_{n=1}^{\infty}f(n)q^n = \lim_{z\rra 0}z\inv \pwr{\frac{1}{\varphi_{\infty}(zf ; q)} - 1}.
\end{equation*}
It is obvious that if $f$ is completely multiplicative, then $\prod_{\lambda_i \in \lambda} f(\lambda_i)=f(n_{\lambda})$, where $n_{\lambda}$ is the so-called "integer" of $\lambda$ defined above. We record one more, obvious consequence of Theorem \ref{1.1}, as we assume it throughout this paper. As before, let $X \subseteq \mathbb Z^+$, and take $\mathcal P_X \subseteq \mathcal P$ to be the set of partitions into elements of $X$. Then clearly by setting $f(n)=0$ if $n\notin X$ in Theorem \ref{1.1}, we see
\begin{equation*}
\frac {1}{\prod_{n\in X} (1-f(n)q^n)} =  \sum_{\lambda \in \mathcal P_X} q^{|\lambda|} \prod_{\lambda_i \in \lambda} f(\lambda_i).
\end{equation*}
The remaining summations in the theorem (aside from \eqref{7}, which may not converge) are taken over $n\in X$.

We see from Theorem \ref{1.1} that we may pass freely between the shapes \eqref{4} -- \eqref{8}, which specialize to a number of classical expressions. For example, taking $f\equiv 1$ in the theorem gives the following fact.

\begin{coro}\label{1.2}
The partition generating function \eqref{1} is true.
\end{coro}

Assuming $\Real (s)>1$, if we take $q = 1$, $f(n) = 1/n^s$ if $n$ is prime and $=0$ otherwise, then Theorem \ref{1.1} yields another classical fact, plus a formula giving the zeta function as a sum over primes.

\begin{coro}\label{1.3}
The Euler product formula \eqref{2} for the zeta function is true. We also have the identity 
\begin{equation*}
\zeta(s) = 1 + \sum_{p \in \mathbb P}\frac{1}{p^s \prod_{r\in \mathbb P,\,r\leq p}\pwr{1 - \frac{1}{r^s}}}.
\end{equation*}
\end{coro}
\vspace{2em}

\section{Partition-theoretic zeta functions}
A multitude of nice specializations of Theorem \ref{1.1} may be obtained. We would like to focus on an interesting class of partition sums arising from Euler's sine function \eqref{3} combined with Theorem \ref{1.1}. Taking $q = 1$ (as we have done in Corollary \ref{1.3}), we begin by noting an easy partition-theoretic formula that may be used to compute the value of $\pi$.

Let $\mathcal P_{m\mathbb Z} \subseteq \mathcal P$ denote the set of partitions into multiples of $m$. Recall from above that the ``integer" $n_{\lambda}$ of a partition $\lambda$ is the product $\lambda_1\lambda_2\cdots \lambda_r$ of its parts.

\begin{coro}\label{1.4}
Summing over partitions into even parts, we have the formula
\begin{equation*}
\frac{\pi}{2} = \sum_{\lambda \in \mathcal P_{2\mathbb Z}} \frac{1}{n_{\lambda}^2}.
\end{equation*}
\end{coro}
We notice that the form of the sum of the right-hand side resembles $\zeta(2)$. Based on this similarity, we wonder if there exists a nice partition-theoretic analog of $\zeta(s)$ possessing something of a familiar zeta function structure---perhaps Corollary \ref{1.4} gives an example of such a function? However, in this case it is not so: the above identity arises from different types of phenomena from those associated with $\zeta(s)$. We have an infinite family of formulas of the following shapes.
\begin{coro}\label{1.5}
Summing over partitions into multiples of any whole number $m > 1$, we have
\begin{eqnarray}
\sum_{\lambda \in \mathcal P_{m\mathbb Z}} \frac{1}{n_{\lambda}^2} &=& \frac{\pi}{m \, \sin \pwr{\frac{\pi}{m}}} \\
\sum_{\lambda \in \mathcal P_{m\mathbb Z}} \frac{1}{n_{\lambda}^4} &=& \frac{\pi^2}{m^2 \, \sin \pwr{\frac{\pi}{m}} \sinh \pwr{\frac{\pi}{m}}},
\end{eqnarray}
and increasingly complicated formulas can be computed for $\sum_{\lambda \in \mathcal P_{m \mathbb Z}} 1/n_{\lambda}^{2^t}$, $t \in \mathbb Z^+$.
\end{coro}

Examples like these are appealing, but their right-hand sides are not entirely reminiscent of the Riemann zeta function, aside from the presence of $\pi$. Certainly they are not as tidy as expressions of the form $\zeta(2k) = ``\pi^{2k} \times \text{rational}"$. Based on the previous corollaries, a reasonable first guess at a partition-theoretic analog of $\zeta(s)$ might be to define
\begin{equation*}
\zeta_{\mathcal P}(s) := \sum_{\lambda \in \mathcal P} \frac{1}{n_{\lambda}^s} = \frac{1}{\prod_{n=1}^{\infty} \pwr{1 - \frac{1}{n^s}}}, \Real (s) > 1.
\end{equation*}
Of course, neither side of this identity converges, but we do obtain convergent expressions if we omit the first term and perhaps subsequent terms of the product to yield $\zeta_{\mathcal P_{\geq a}}(s) := \sum_{\lambda \in \mathcal P_{\geq a}} 1/n_{\lambda}^s = \prod_{n=a}^{\infty} (1-1/n^s)\inv\,\,\,(a\geq 2)$, where $\mathcal P_{\geq a} \subset \mathcal P$ denotes the set of partitions into parts greater than or equal to $a$. For instance, we have the following formula.
\begin{coro}\label{1.6}
Summing over partitions into parts greater than or equal to $2$, we have 
\begin{equation*}
\zeta_{\mathcal P_{\geq 2}}(3) = \sum_{\lambda \in \mathcal P_{\geq 2}} \frac{1}{n_{\lambda}^3} = \frac{3\pi }{\cosh \pwr{\frac{1}{2} \pi \sqrt{3}}}.
\end{equation*}
\end{coro}

While it is an interesting expression, stemming from an identity of Ramanujan \cite{ramanujan2000collected}, once again this formula does not seem to evoke the sort of structure we anticipate from a zeta function---of course, the value of $\zeta(3)$ is not even known. We need to find the ``right" subset of $\mathcal P$ to sum over, if we hope to find a nice partition-theoretic zeta function. As it turns out, there are subsets of $\mathcal P$ that naturally produce analogs of $\zeta(s)$ for certain values of $s$.

\begin{defn}
We define a partition-theoretic generalization $\zeta_{\mathcal P}(\brk{s}^k)$ of the Riemann zeta function by the following sum over all partitions $\lambda$ of fixed length $\ell(\lambda) = k \in \mathbb Z_{\geq 0}$ at argument $s \in \mathbb C$, $\Real (s) > 1$:
\begin{equation}\label{11}
\zeta_{\mathcal P}(\brk{s}^k) := \sum_{\ell(\lambda) = k} \frac{1}{n_{\lambda}^s}. 
\end{equation}
\end{defn}

\begin{remark}
This is a fairly natural formation, being similar in shape (and notation) to the weight $k$ multiple zeta function $\zeta(\brk{s}^k)$, which is instead summed over length-$k$ partitions into distinct parts; Hoffman gives interesting formulas relating $\zeta_{\mathcal P}(\brk{s}^k)$ (in different notation) to combinations of multiple zeta functions \cite{hoffman1992multiple}, which exhibit rich structure.
\end{remark}

We have immediately that $\zeta_{\mathcal P}(\brk{s}^0) = 1/n_{\emptyset}^s = 1$ and  $\zeta_{\mathcal P}(\brk{s}^1) = \zeta_{\mathcal P}(\brk{s}) = \zeta(s)$. Using Theorem \ref{1.1} and proceeding (see Section \ref{Proofs of theorems and corollaries}) much as Euler did to find the value of $\zeta(2k)$ \cite{dunham1999euler}, we are able to find explicit values for $\zeta_{\mathcal P}(\brk{2}^k)$ at every positive integer $k >0$. Somewhat surprisingly, we find that in these cases $\zeta_{\mathcal P}(\brk{2}^k)$ is a rational multiple of $\zeta(2k)$.

\begin{coro}\label{1.7}
For $k>0$, we have the identity 
\begin{equation*}
\zeta_{\mathcal P}(\brk{2}^k) = \sum_{\ell(\lambda) = k} \frac{1}{n_{\lambda}^2} = \frac{2^{2k - 1} - 1}{2^{2k-2}}\zeta(2k).
\end{equation*}
For example, we have the following values:
\begin{align*}
\zeta_{\mathcal P}(\brk{2}) &= \zeta(2) = \frac{\pi^2}{6}, \\
\zeta_{\mathcal P}(\brk{2}^2) &= \frac{7}{4}\zeta(4) = \frac{7\pi^4}{360}, \\
\zeta_{\mathcal P}(\brk{2}^3) &= \frac{31}{16}\zeta(6)= \frac{31\pi^6}{15120},\dots , \\
\zeta_{\mathcal P}(\brk{2}^{13}) &= \frac{33554431}{16777216}\zeta(26) = \frac{22076500342261\pi^{26}}{93067260259985915904000000},\dots
\end{align*}
\end{coro}

Corollary \ref{1.7} reveals that $\zeta_{\mathcal P}(\brk{2}^k)$ is indeed of the form $``\pi^{2k} \times \text{rational}"$ for all positive $k$, like the zeta values $\zeta(2k)$ given by Euler (we note that $\zeta(26)$ is the highest zeta value Euler published) \cite{dunham1999euler}. We have more: we can find $\zeta_{\mathcal P}(\brk{2^t}^k)$ explicitly for all $t \in \mathbb Z^+$. These values are finite combinations of well-known zeta values, and are also of the form $``\pi^{2^tk} \times \text{rational}"$.

\begin{coro}\label{1.8}
For $k>0$ we have the identity
\begin{align*}
\zeta_{\mathcal P}(\brk{4}^k) & = \sum_{\ell(\lambda) = k}\frac{1}{n_{\lambda}^4}\\ 
& = \frac{1}{16^{k-1}} \pwr{\sum_{n=0}^{2k} (-1)^n(2^{2n-1}-1)(2^{4k - 2n - 1} - 1)\zeta(2n)\zeta(4k - 2n)},
\end{align*}
and increasingly complicated formulas can be computed for $\zeta_{\mathcal P}(\brk{2^t}^k), t \in \mathbb Z^+$.
\end{coro}
\begin{remark}
The summation on the far right above may be shortened by noting the symmetry of the summands around the $n=k$ term.
\end{remark}

It would be desirable to understand the value of $\zeta_{\mathcal P}(\brk{s}^k)$ at other arguments $s$; the proof we give below (see Section \ref{Proofs of theorems and corollaries}) does not shed much light on this question, being based very closely on Euler's formula \eqref{1.3}, which forces $s$ be a power of $2$. Also, if we solve Corollary \ref{1.6} for $\zeta(0)$, we conclude that $\zeta(0) = \frac{2^{-2}}{2\inv - 1}\zeta_{\mathcal P}(\brk{2}^0) = -1/2$, which is the value of $\zeta(0)$ under analytic continuation. Can $\zeta_{\mathcal P}(\brk{s}^k)$ be extended via analytic continuation for values of $k > 1$? In a larger sense we wonder: do nice zeta function analogs exist if we sum over other interesting subsets of $\mathcal P$?

We do have a few general properties shared by convergent series $\sum 1/n_{\lambda}^s$ summed over large subclasses of $\mathcal P$. First we need to refine some of our previous notations.

\begin{defn}
Take any subset of partitions $\mathcal P' \subseteq \mathcal P$. Then for $\Real (s) > 1$, on analogy to classical zeta function theory, when these expressions converge we define
\begin{align}
\zeta_{\mathcal P'} &:= \sum_{\lambda \in \mathcal P'} \frac{1}{n_{\lambda}^s}, &\eta_{\mathcal P'}(s)& := \sum_{\lambda \in \mathcal P'} \frac{(-1)^{\ell(\lambda)}}{n_{\lambda}^s}, &\zeta_{\mathcal P'}(\brk{s}^k) &:= \sum_{\begin{tiny}
\begin{array}{c}
\lambda \in \mathcal P' \\
\ell(\lambda) = k
\end{array}\end{tiny}} \frac{1}{n_{\lambda}^s}.
\end{align}
\end{defn}

\begin{remark} 
As important special cases, we have $\zeta_{\mathcal P_{\mathbb P}}(s) = \zeta(s)$ and $\zeta_{\mathcal P_{\mathbb Z^+}}(\brk{s}^k) = \zeta_{\mathcal P}(\brk{s}^k)$. It is also easy to see that $\zeta_{\mathcal P'}(s)=\sum_{k=0}^{\infty} \zeta_{\mathcal P'}(\brk{s}^k)$ and $\eta_{\mathcal P'}(s)=\sum_{k=0}^{\infty} (-1)^{k} \zeta_{\mathcal P'}(\brk{s}^k)$ if we assume absolute convergence. Moreover, given absolute convergence, we may write $\zeta_{\mathcal P'}(s),\,\zeta_{\mathcal P'}(\brk{s}^k)$ as classical Dirichlet series related to multiplicative partitions: we have $\zeta_{\mathcal P'}(s)=\sum_{j=1}^{\infty}\# \{\lambda \in \mathcal P'\,|\,\, n_{\lambda}=j\}\, j^{-s}$ and $\zeta_{\mathcal P'}(\brk{s}^k)(s)=\sum_{j=1}^{\infty}\# \{\lambda \in \mathcal P'\,|\,\, \ell(\lambda)=k,n_{\lambda}=j\}\,j^{-s}$ (see \cite{chamberland2013gamma} for more about multiplicative partitions).

\end{remark} 

As previously, take $X \subseteq \mathbb Z^+$ and take $\mathcal P_{X} \subseteq \mathcal P$ to denote partitions into elements of $X$ (thus $\mathcal P_{\mathbb Z^+} = \mathcal P$). Note that $\zeta_{\mathcal P_X}(s) = \prod_{n \in X} \pwr{1 - \frac{1}{n^s}}\inv$ is divergent if $1 \in X$ and, when $X$ is finite (thus there is no restriction on the value of $\Real(s)$), if $s=i\pi j/\log n$ for any $n\in X$ and even integer $j$. Similarly, when $X$ is finite, $\eta_{\mathcal P_X}(s) = \prod_{n \in X} \pwr{1 + \frac{1}{n^s}}\inv$ is divergent if $s=i\pi k/\log n$ for any $n\in X$ and odd integer $k$. Clearly if $Y \subseteq \mathbb Z^+$, then from the product representations we also have $\zeta_{\mathcal P_X}(s)\zeta_{\mathcal P_Y}(s)=\zeta_{\mathcal P_{X \cup Y}}(s)\zeta_{\mathcal P_{X \cap Y}}(s)$ and $\eta_{\mathcal P_X}(s)\eta_{\mathcal P_Y}(s)=\eta_{\mathcal P_{X \cup Y}}(s)\eta_{\mathcal P_{X \cap Y}}(s)$. 

Many interesting subsets of partitions have the form $\mathcal P_X$, in particular those to which Theorem \ref{1.1} most readily applies. Note that such subsets $\mathcal P_X$ are partition ideals of order 1, in the sense of Andrews \cite{andrews1998theory}. With the above notations, we have the following useful ``doubling" formulas.

\begin{coro}\label{1.9}
If $\zeta_{\mathcal P_X}(s)$ converges over $\mathcal P_X \subseteq \mathcal P$, then 
\begin{equation}\label{13}
\zeta_{\mathcal P_X}(2s) = \zeta_{\mathcal P_X}(s)\eta_{\mathcal P_X}(s).
\end{equation}
Furthermore, for $n \in \mathbb Z_{\geq 0}$ we have the identity 
\begin{equation}\label{14}
\zeta_{\mathcal P_X}\left(\{2^{n+1}s\}^k\right) = \sum_{j=0}^{2^nk}(-1)^j\zeta_{\mathcal P_X}\left(\brk{2^ns}^j\right)\zeta_{\mathcal P_X}\left(\brk{2^ns}^{2^nk - j}\right).
\end{equation}
\end{coro}
\begin{remark}
As in Corollary \ref{1.8}, the summation on the right-hand side of \eqref{14} may be shortened by symmetry.
\end{remark}

If we take $X = \mathbb P$, then \eqref{13} reduces to the classical identity $\zeta(2s) = \zeta(s)\sum_{n=1}^{\infty} \lambda(n)/n^s$, where $\lambda(n)$ is Liouville's function. Another specialization of Corollary \ref{1.9} leads to new information about $\zeta_{\mathcal P}(\brk{s}^k)$: we may extend the domain of $\zeta_{\mathcal P}(\brk{s}^k)$ to $\Real (s) > 1$ if we take $X = \mathbb Z^+$, $n = 0$, $k = 2$. We find $\zeta_{\mathcal P}(\brk{s}^2)$ inherits analytic continuation from the sum on the right-hand side below.

\begin{coro}\label{1.10}
For $\Real (s) > 1$, we have 
\begin{equation*}
\zeta_{\mathcal P}(\brk{s}^2) = \frac{\zeta(2s) + \zeta(s)^2}{2}.
\end{equation*}
\end{coro}

\begin{remark}
This resembles a well-known series shuffle product formula for multiple zeta values \cite{besser416double}.
\end{remark}

Another interesting subset of $\mathcal P$ is the set of partitions $\mathcal P^*$ into \textit{distinct} parts; also of interest is the set of partitions $\mathcal P_X^*$ into distinct elements of $X \subseteq \mathbb Z^+$ (thus $\mathcal P_{\mathbb Z^+}^* = \mathcal P^*$). However, partitions into distinct parts are not immediately compatible with the identities in Theorem \ref{1.1}. Happily, we have a dual theorem that leads us to zeta functions summed over $\mathcal P_X^*$ for any $X\subseteq \mathbb Z^+$.

Let us recall the infinite product $\varphi_{\infty}(f;q)$ from Theorem \ref{1.1}.

\begin{theorem}\label{1.11}
If the product converges, then $\varphi_{\infty}(f;q) = \prod_{n=1}^{\infty}(1 - f(n)q^n)$ may be expressed in a number of equivalent forms, viz.
\begin{align}
\varphi_{\infty}(f;q) &= \sum_{\lambda \in \mathcal P^*}(-1)^{\ell(\lambda)} q^{|\lambda|} \prod_{\lambda_i \in \lambda}f(\lambda_i) \label{15}\\ \displaybreak
&= \begin{array}{c}
1 - \sum_{(6)}
\end{array} \label{16}\\
& = \begin{array}{c}
1 - \varphi_{\infty}(f;q) \sum_{(5)} 
\end{array} \label{17}\\
& = 1-\cfrac{\sum_{(5)}}{1+\cfrac{\sum_{(6)}}{1-\cfrac{\sum_{(5)}}{1+\cfrac{\sum_{(6)}}{1- \cdots}}}} \label{18}
\end{align}
where $\sum_{(5)},\sum_{(6)}$ are exactly as in Theorem \ref{1.1}, and the sum in \eqref{15} is taken over the partitions into distinct parts.
\end{theorem}

\begin{remark}
Note that there is not a nice ``inverted" sum of the form \eqref{7} here.
\end{remark}

Just as with Theorem \ref{1.1}, we may write arbitrary power series as limiting cases, and we have the obvious identity
\begin{equation*}
\prod_{n\in X} (1-f(n)q^n) =  \sum_{\lambda \in \mathcal P_X^*} (-1)^{\ell(\lambda)}q^{|\lambda|} \prod_{\lambda_i \in \lambda} f(\lambda_i),
\end{equation*}
with the remaining summations in Theorem \ref{1.11} being taken over elements of $X$. For completeness, we record another obvious but useful consequence of Theorems \ref{1.1} and \ref{1.11}. The following statement might be viewed as a generalized eta quotient formula, with coefficients given explicitly by finite combinatorial sums.

\begin{coro}\label{etaquotient}
For $f_j$ defined on $X_j\subseteq \mathbb Z^+$, consider the double product
\[
\prod_{j=1}^{n}\prod_{k_j\in X_j}\left(1\pm f_j(k_j)q^{k_j}\right)^{\pm 1}
=\sum_{k=0}^{\infty}c_k q^k
,\]
where the $\pm$ sign is fixed for fixed $j$, but may vary as $j$ varies. Then the coefficients $c_k$ are given by the $(n-1)$-tuple sum
\begin{multline*}
c_k= \sum_{k_2=0}^{k}\sum_{k_3=0}^{k_2}\dots\sum_{k_{n}=0}^{k_{n-1}}\left(\sum_{\substack{\lambda\vdash k_{n}\\ \lambda\in\mathcal P_{X_{n}}^{\pm}}}\prod_{\lambda_i\in\lambda}f_{n}(\lambda_i)\right)\left(\sum_{\substack{\lambda\vdash (k_{n-1}-k_{n})\\ \lambda\in\mathcal P_{X_{n-1}}^{\pm}}}\prod_{\lambda_i\in\lambda}f_{n-1}(\lambda_i)\right)\\
\times\left(\sum_{\substack{\lambda\vdash (k_{n-2}-k_{n-1})\\ \lambda\in\mathcal P_{X_{n-2}}^{\pm}}}\prod_{\lambda_i\in\lambda}f_{n-2}(\lambda_i)\right)\dots\left(\sum_{\substack{\lambda\vdash (k-k_2)\\ \lambda\in\mathcal P_{X_{1}}^{\pm}}}\prod_{\lambda_i\in\lambda}f_1(\lambda_i)\right)
\end{multline*}
in which we have set $\mathcal P_{X_j}^-:=\mathcal P_{X_j}$ and $\mathcal P_{X_j}^+:=\mathcal P_{X_j}^*$ with the $\pm$ sign as associated to each $j$ above.
\end{coro}

\begin{remark} The $+$ or $-$ signs in the formula for $c_k$ indicate partitions arising from the numerator or denominator, respectively, of the double product. One may replace $f_j$ with $-f_j$ to effect further sign changes.
\end{remark}

Analogous corollaries to those following Theorem \ref{1.1} are available, but we wish right away to apply this theorem to the problem at hand, the investigation of partition zeta functions. We have $\zeta_{\mathcal P_X^*}(s) = \prod_{n \in X} (1 + \frac{1}{n^s})$ as well as $\eta_{\mathcal P_X^*}(s) = \prod_{n \in X} (1 - \frac{1}{n^s})$. It is immediate then from \eqref{15} that for $\Real (s) > 1$ we also have the following relations, where the sum on the left-hand side of each equation is taken over the partitions into distinct elements of $X$:
\begin{align}
\zeta_{\mathcal P_X^*}(s) & = \frac{1}{\eta_{\mathcal P_X}(s)}, &\eta_{\mathcal P_X^*}(s)& = \frac{1}{\zeta_{\mathcal P_X}(s)} \label{19}
\end{align}
Note that $\zeta_{\mathcal P_X^*}(s)$ and $\eta_{\mathcal P_X^*}(s)$ are finite sums (and entire functions of $s$) if $X$ is a finite set, unlike $\zeta_{\mathcal P_X}(s)$ and $\eta_{\mathcal P_X}(s)$. Note also that $\eta_{\mathcal P_X^*}(s) = 0$ identically if $1 \in X$, with zeros when $X$ is finite at the values $s=i\pi j/\log n$ for any $n\in X$ and $j$ even. Unlike $\zeta_{\mathcal P}(s)$, we can see from \eqref{19} that $\zeta_{\mathcal P^*}(s)$ is well-defined on $\Real (s) > 1$ (thus both $\zeta_{\mathcal P_X^*}$ and $\eta_{\mathcal P_X^*}$ are well-defined over all subsets $\mathcal P_X^*$ of $\mathcal P^*$); when $X$ is finite, $\zeta_{\mathcal P^*}(s)$ has zeros at $s=i\pi k/\log n$ for $n\in X$ and $k$ odd. Morever, we have $\zeta_{\mathcal P_X^*}(s)\zeta_{\mathcal P_Y^*}(s)=\zeta_{\mathcal P_{X \cup Y}^*}(s)\zeta_{\mathcal P_{X \cap Y}^*}(s)$ and $\eta_{\mathcal P_X^*}(s)\eta_{\mathcal P_Y^*}(s)=\eta_{\mathcal P_{X \cup Y}^*}(s)\eta_{\mathcal P_{X \cap Y}^*}(s)$. Here is an example of a zeta sum of this form. 

\begin{coro}\label{1.12}
Summing over partitions into distinct parts, we have that 
\begin{equation*}
\zeta_{\mathcal P^*}(2) = \sum_{\lambda \in \mathcal P^*} \frac{1}{n_{\lambda}^2} = \frac{\sinh \pi }{\pi}.
\end{equation*}
\end{coro}

Zeta sums over partitions into distinct parts admit an important special case: as we remarked beneath definition \eqref{11}, the multiple zeta function $\zeta(\brk{s}^k)$ can be written 
\begin{equation}\label{20}
\zeta(\brk{s}^k) := \sum_{\lambda_1 > \lambda_2 > \cdots > \lambda_k \geq 1} \frac{1}{\lambda_1^s\lambda_2^s\cdots \lambda_k^s} = \sum_{\begin{tiny}
\begin{array}{c}
\lambda \in \mathcal P^* \\
\ell(\lambda) = k
\end{array}\end{tiny}} \frac{1}{n_{\lambda}^s} = \zeta_{\mathcal P^*}(\brk{s}^k).
\end{equation}
Using this notation, we can derive even simpler formulas for the multiple zeta values $\zeta(\brk{2^t}^k)$ than those found for $\zeta_{\mathcal P}(\brk{2^t}^k)$ in Corollaries \ref{1.7} and \ref{1.8}. For instance, we have the following values.

\begin{coro}\label{1.13}
For $k>0$ we have the identities
\begin{multline*}
\zeta(\brk{2}^k) = \frac{\pi^{2k}}{(2k+1)!}, \\
\zeta(\brk{4}^k) = \pi^{4k} \sum_{n=0}^{2k} \frac{(-1)^n}{(2n+1)!(4k-2n+1)!}, \\
\zeta(\brk{8}^k) = \pi^{8k} \sum_{n=0}^{4k}(-1)^n \pwr{\sum_{i=0}^n \frac{(-1)^i}{(2i+1)!(2n-2i +1)!}}\\
\times\pwr{\sum_{i=0}^{4k-n} \frac{(-1)^i}{(2i+1)!(8k-2n-2i+1)!}},
\end{multline*}
and increasingly complicated formulas of the shape $``\pi^{2^tk} \times \text{finite sum of fractions}"$ can be computed for multiple zeta values of the form $\zeta(\brk{2^t}^k),$ $t \in \mathbb Z^+$.
\end{coro}

\begin{remark}
The first identity above is proved in \cite{hoffman1992multiple} by a different approach from that taken here (see Section \ref{Proofs of theorems and corollaries}); it is possible the other identities in the corollary are also known.
\end{remark}

The summations in Corollary \ref{1.13} arise from quite general properties: we have these ``doubling" formulas comparable to Corollary \ref{1.9}.

\begin{coro}\label{1.14}
If $\zeta_{\mathcal P_X^*}(s)$ converges over $\mathcal P_X^* \subseteq \mathcal P$, then 
\begin{equation}\label{21}
\eta_{\mathcal P_X^*}(2s) = \eta_{\mathcal P_X^*}(s)\zeta_{\mathcal P_X^*}(s).
\end{equation}
Furthermore, for $n \in \mathbb Z_{\geq 0}$ we have
\begin{equation}\label{22}
\zeta_{\mathcal P_X^*}\left(\{2^{n+1}s\}^k\right) = \sum_{j=0}^{2^nk}(-1)^j \zeta_{\mathcal P_X^*}\left(\brk{2^ns}^j\right)\zeta_{\mathcal P_X^*}\left(\brk{2^ns}^{2^nk-j}\right).
\end{equation}
\end{coro}

\begin{remark}
Once again, the summation on the right-hand side of \eqref{22} may be be shortened by symmetry. Equation \eqref{22} yields a family of multiple zeta function identities when we let $X = \mathbb Z^+$. 
\end{remark}

We note that by recursive arguments, from \eqref{13} and \eqref{21} together with \eqref{8}, we have these curious product formulas connecting sums over partitions into distinct parts to their counterparts involving unrestricted partitions:
\begin{align*}
\zeta_{\mathcal P_X^*}(s)\zeta_{\mathcal P_X^*}(2s)\zeta_{\mathcal P_X^*}(4s)\zeta_{\mathcal P_X^*}(8s)\cdots &= \zeta_{\mathcal P_X}(s) \\
\eta_{\mathcal P_X}(s)\eta_{\mathcal P_X}(2s)\eta_{\mathcal P_X}(4s)\eta_{\mathcal P_X}(8s)\cdots &= \eta_{\mathcal P_X^*}(s)
\end{align*}

Now, if we take $X = \mathbb P$ then \eqref{21} becomes the classical identity $\zeta(2s)\inv = \zeta(s)\inv \sum_{n=1}^{\infty} |\mu(n)|/n^s$, where $\mu(n)$ is the M\"obius function. We might view the simple quantity $(-1)^{\ell(\lambda)}$ as a partition-theoretic generalization of $\mu$; it specializes to the M\"obius function (when considering partitions into distinct prime parts), and also to Liouville's function (considering unrestricted prime partitions), as we saw above. K. Alladi has observed this correspondence as well (personal communication, December 22, 2015). 

It is fascinating---and rather mysterious---that partitions (which are defined additively, with no connection to multiplication) into parts that are prime numbers (defined multiplicatively) should have significant number-theoretic connections.

The literature abounds with product formulas which, when fed through the machinery of the identities noted here, produce nice partition zeta sum variants; the interested reader is referred to \cite{chamberland2013gamma} as a starting point for further study.

\vspace{2em}

\section{Proofs of theorems and corollaries}\label{Proofs of theorems and corollaries}

\begin{proof}[Proof of Theorem \ref{1.1}]
Identity \eqref{4} appears in a different form as \cite[Equation~22.16]{fine1988basic}. The proof proceeds formally, much like the standard proof of \eqref{1.1}; we expand $1/ \varphi_{\infty}(f;q)$ as a product of geometric series
\begin{multline*}
\frac{1}{\varphi_{\infty}(f;q)} = (1+ f(1)q + f(1)^2q^2 + f(1)^3q^3 + \dots)\\
\times(1 + f(2)q^2 + f(2)^2q^4 + f(2)^3q^6 + \dots)\times\cdots
\end{multline*}
and multiply out all the terms (without collecting coefficients in the usual way). The result is the partition sum in \eqref{4}. 

Identities \eqref{5} and \eqref{6} are proved using telescoping sums. Consider that 
\begin{align*}
\frac{1}{\varphi_{\infty}(f;q)} &= \frac{1}{\varphi_0(f;q)} + \sum_{n=1}^{\infty} \pwr{\frac{1}{\varphi_n(f;q)} - \frac{1}{\varphi_{n-1}(f;q)}} \\
&= 1 + \sum_{n=1}^{\infty} \frac{1}{\varphi_{n-1}(f;q)} \pwr{\frac{1}{1- f(n)q^n} - 1} \\
&= 1 + \sum_{n=1}^{\infty}q^n \frac{f(n)}{\varphi_n(f;q)} = \begin{array}{c}
1 + \sum_{(5)}
\end{array},
\end{align*}
recalling the notation $\sum_{(5)}$ (as well as $\sum_{(6)}$) from the theorem, which is \eqref{5}. Similarly, we can show
\begin{align*}
\varphi_{\infty}(f;q) &= \varphi_0(f;q) + \sum_{n=1}^{\infty} (\varphi_n(f;q) - \varphi_{n-1}(f;q)) \\
&= 1 - \sum_{n=1}^{\infty}q^nf(n) \varphi_{n-1}(f;q) = \begin{array}{c}
1- \sum_{(6)}
\end{array}.
\end{align*}
Thus we have 
$$\begin{array}{c}
\sum_{(5)}
\end{array} = \frac{1}{\varphi_{\infty}(f; q)} - 1 = \frac{1 - \varphi_{\infty}(f;q)}{\varphi_{\infty}(f;q)} = \frac{\sum_{(6)}}{\varphi_{\infty}(f;q)},$$
which leads to \eqref{6}.

The proof of \eqref{7} is similar to the proof we gave of \cite[Theorem~1.1(1)]{rolen2013strange}. Substitute the identity
$$\varphi_n(f;q) = \prod_{k=1}^{n} (1-f(k)q^k) = (-1)^nq^{n(n+1)/2}\varphi_n(1/f ; q\inv)\prod_{k=1}^n f(k)$$
term-by-term into the sum \eqref{5} and simplify to find the desired expression.

The proof of \eqref{8} is inspired by the standard proof of the continued fraction representation of the golden ratio. It follows from the proof above of \eqref{5} and \eqref{6} that 
\begin{align*}
\frac{1}{\varphi_{\infty}(f;q)} &= 1 + \frac{\sum_{(6)}}{\varphi_{\infty}(f;q)} \\
&= 1 + \frac{\sum_{(6)}}{1-\varphi_{\infty}(f;q)\sum_{(5)}} \\
&= 1+\cfrac{\sum_{(6)}}{1-\cfrac{\sum_{(5)}}{1/\varphi_{\infty}(f;q)}}.
\end{align*}
We notice that the expression on the left-hand side is also present on the far right in the denominator. We replace this term $1/\varphi_{\infty}(f;q)$ in the denominator with the entire right-hand side of the equation; reiterating this process indefinitely gives \eqref{8}.
\end{proof}
\begin{remark}
The series$\begin{array}{c}
\sum_{(5)},\sum_{(6)}
\end{array}$enjoy other nice, golden ratio-like relationships. For instance, because $$\begin{array}{c}
(1 + \sum_{(5)})(1-\sum_{(6)}) = 1 
\end{array},$$ it is easy to see that 
$$\begin{array}{c}
\sum_{(5)} - \sum_{(6)} = \sum_{(5)}\sum_{(6)},
\end{array}$$
which resembles the formula $\phi - 1/\phi = \phi \cdot 1/\phi$ involving the golden ratio $\phi$ and its reciprocal.
\end{remark}

\begin{proof}[Proof of Corollary \ref{1.2}]
This is immediate upon letting $f\equiv 1$ in \eqref{4}, as
$$\sum_{\lambda \in \mathcal P}q^{|\lambda|}\prod_{\lambda_i \in \lambda}f(\lambda_i) = 1 + \sum_{n=1}^{\infty}q^n\sum_{\lambda \vdash n}\prod_{\lambda_i \in \lambda} f(\lambda_i).$$
\end{proof}

\begin{proof}[Proof of Corollary \ref{1.3}]
As noted above, we assume $\Real (s)>1$. Let $q = 1$, $f(n) = 1/n^s$ if $n$ is prime and $=0$ otherwise; then by \eqref{4}
$$\frac{1}{\prod_{p \in \mathbb P}\pwr{1-\frac{1}{p^s}}} = \sum_{\lambda \in \mathcal P_{\mathbb P}} \frac{1}{n_{\lambda}^s}.$$
Consider the prime decomposition of a positive integer $n = p_1^{a_1}p_2^{a_2}\cdots p_r^{a_r}$, $p_1>p_2>\cdots > p_r$. We will associate this decomposition to the unique partition into prime parts $\lambda = (p_1,\dots , p_1,p_2,\dots,$ $p_2,\dots ,p_r,\dots ,p_r) \in \mathcal P$, where $p_k \in \mathbb P$ is repeated $a_k$ times (thus $n$ is equal to $n_{\lambda}$). Every positive integer $n\geq 1$ is associated to exactly one partition into prime parts (with $n = 1$ associated to $\emptyset \in \mathcal P_{\mathbb P}$), and conversely: there is a bijective correspondence between $\mathbb Z^+$ and $\mathcal P_{\mathbb P}$ (Alladi and Erd\H{o}s give an interesting study \cite{alladi1977additive} along these lines). Therefore we see by absolute convergence that 
$$\sum_{n\geq 1}\frac{1}{n^s} = \sum_{\lambda \in \mathcal P_{\mathbb P}} \frac{1}{n_{\lambda}^s}.$$
Equating the left-hand sides of the above two identities gives Euler's product formula \eqref{2}. The series given for $\zeta(s)$ follows immediately from Theorem \eqref{5} with the above definition of $f$.
\end{proof}

\begin{proof}[Proof of Corollary \ref{1.4}]
This is actually a special case of the subsequent Corollary \ref{1.5}, setting $m=2$ in the first equation (see below).
\end{proof}

\begin{proof}[Proof of Corollary \ref{1.5}]
We begin with an identity equivalent to \eqref{3} and its ``$+$" companion:
\begin{align*}
\frac{\pi z}{\sin (\pi z)} &= \frac{1}{\prod_{n=1}^{\infty} \pwr{1- \frac{z^2}{n^2}}}, &\frac{\pi z}{\sinh (\pi z)}& = \frac{1}{\prod_{n=1}^{\infty} \pwr{1+ \frac{z^2}{n^2}}}
\end{align*}
If $\omega_k := e^{2\pi i /k}$, then $\omega_{2k}^2 = \omega_k$ and we have, by multiplying the above two identities, the pair
$$
\frac{\pi^2z^2}{\sin (\pi z)\sinh (\pi z)} = \frac{1}{\prod_{n=1}^{\infty} \pwr{1-\frac{z^4}{n^4}}},\  \  \frac{\omega_4\pi^2z^2}{\sin (\omega_8\pi z)\sinh (\omega_8\pi z)} = \frac{1}{\prod_{n=1}^{\infty} \pwr{1+\frac{z^4}{n^4}}}.
$$
Multiplying these two equations, and repeating this procedure indefinitely, we find identities like
\begin{align*}
\frac{\omega_4\pi^4z^4}{\sin (\pi z)\sinh(\pi z)\sin(\omega_8\pi z)\sinh(\omega_8 \pi z)} &= \frac{1}{\prod_{n=1}^{\infty} \pwr{1 - \frac{z^8}{n^8}}}, 
\end{align*}
\begin{multline*}
\frac{\omega_4^2\pi^8z^8}{\sin(\pi z)\sinh(\pi z) \sin (\omega_8 \pi z)\sinh (\omega_8 \pi z) }\\
\times\frac{1}{\sin(\omega_{16}\pi z) \sinh (\omega_{16}\pi z)\sin (\omega_8\omega_{16}\pi z) \sinh (\omega_8\omega_{16} \pi z)}
\end{multline*}
$$
\quad \quad \quad = \frac{1}{\prod_{n=1}^{\infty} \pwr{1 - \frac{z^{16}}{n^{16}}}},
$$
as well as their ``$+$" companions, and so on. On the other hand, it follows from \eqref{4} that 
$$\frac{1}{\prod_{n=1}^{\infty} \pwr{1 - \frac{zq^n}{n^s}}} = \sum_{\lambda \in \mathcal P}q^{|\lambda|} \prod_{\lambda_i \in \lambda} \frac{z}{\lambda_i^s} = \sum_{\lambda \in \mathcal P}\frac{z^{\ell(\lambda)}q^{|\lambda|}}{n_{\lambda}^s}.$$
Replacing $z$ with $\pm z^{2^t}$ and taking $q = 1$ in the above expression, it is easy to see that we have 
\begin{align*}
\frac{1}{\prod_{n=1}^{\infty} \pwr{1 - \frac{z^{2^t}}{n^{2^t}}}} &= \sum_{\lambda \in \mathcal P}\frac{z^{2^t \ell(\lambda)}}{n_{\lambda}^{2^t}}, &\frac{1}{\prod_{n=1}^{\infty} \pwr{1 + \frac{z^{2^t}}{n^{2^t}}}} &= \sum_{\lambda \in \mathcal P}\frac{(-1)^{\ell(\lambda)}z^{2^t \ell(\lambda)}}{n_{\lambda}^{2^t}}.
\end{align*}
These series have closed forms given by complicated trigonometric and hyperbolic expressions such as the ones above. Setting $z = 1/m$ in such expressions yields the explicit values advertised in the corollary for
\begin{align*}
\frac{1}{\prod_{n=1}^{\infty} \pwr{1 - \frac{1}{m^{2^t}n^{2^t}}}} & = \frac{1}{\prod_{n=1}^{\infty} \pwr{1 - \frac{1}{(mn)^{2^t}}}}\\
& = \frac{1}{\prod_{n\equiv 0  \, (\text{mod }m)} \pwr{1 - \frac{1}{n^{2^t}}}} = \sum_{\lambda \in \mathcal P_{m\mathbb Z}} \frac{1}{n_{\lambda}^{2^t}}.
\end{align*}\end{proof}

\begin{remark}
More generally, let $\mathcal P_{a(m)}$ denote the set of partitions into parts $\equiv a \pmod m$ (so $\mathcal P_{m\mathbb Z}$ is $\mathcal P_{0(m)}$ in this notation). It is clear that if $\lambda \in \mathcal P_{a(m)}$ then $n_{\lambda}^s \equiv a^s \pmod m$, thus we find 
$$\frac{1}{\prod_{n\equiv a \, (\text{mod }m)} (1-n^sq^n)}  = \sum_{\lambda \in \mathcal P_{a(m)}} n_{\lambda}^sq^{|\lambda|} \equiv \frac{1}{(a^sq^a ; q^m)_{\infty}} \pmod m .$$
Of course, these expressions diverge as $q \rra 1$ so $\zeta_{\mathcal P_{a(m)}}(-s)$ does not make sense, but we wonder: do there exist similarly nice relations that involve $\zeta_{\mathcal P_{a(m)}}(s)$ or a related form?
\end{remark}

\begin{proof}[Proof of Corollary \ref{1.6}]
We apply \eqref{4} to the following formula submitted by Ramanujan as a problem to the \textit{Journal of the Indian Mathematical Society}, reprinted as \cite[Question~261]{ramanujan2000collected}:
$$\prod_{n=2}^{\infty} \pwr{1 - \frac{1}{n^3}} = \frac{\cosh \pwr{\frac{1}{2} \pi \sqrt{3}}}{3\pi}.$$
Take $q=1$, $f(n) = 1/n^3$ if $n>1$ and $=0$ otherwise in \eqref{4}. Comparing the result with the above formula gives the corollary.
\end{proof}

\begin{remark}
Ramanujan gives a companion formula $\prod_{n=1}^{\infty} \pwr{1 + \frac{1}{n^3}}=\cosh \pwr{\frac{1}{2}\pi \sqrt{3}}/\pi$ in the same problem \cite{ramanujan2000collected}. Multiplying this infinite product by the one above and using \eqref{4} yields a closed form for $\sum_{\lambda \in \mathcal P_{\geq 2}} 1/n_{\lambda}^6$ as well.
\end{remark}

\begin{proof}[Proof of Corollary \ref{1.7}]
Consider the sequence $\beta_{2k}$ of coefficients of the expansion
\begin{equation}\label{23}
\frac{z}{\sinh z} =  \frac{1}{\prod_{n=1}^{\infty} \pwr{1 + \frac{z^2}{\pi^2n^2}}} = \sum_{k=0}^{\infty} \beta_{2k}z^{2k}.
\end{equation}
From the Maclaurin series for the hyperbolic cosecant and Euler's work relating the zeta function to the Bernoulli numbers, it follows that 
\begin{equation}\label{24}
\beta_{2k} = \frac{4(-1)^k(2^{2k-1} - 1)\zeta(2k)}{(2\pi)^{2k}}.
\end{equation}
On the other hand, from \eqref{4} we have
$$\frac{1}{\prod_{n=1}^{\infty} \pwr{1 + \frac{z^2}{\pi^2n^2}}} = \sum_{\lambda \in \mathcal P} \frac{(-1)^{\ell(\lambda)} z^{2\ell(\lambda)}}{\pi^{2\ell(\lambda)}n_{\lambda}^2} = \sum_{k=0}^{\infty} \frac{(-1)^kz^{2k}}{\pi^{2k}}\sum_{\ell(\lambda) = k} \frac{1}{n_{\lambda}^2},$$
thus 
$$\beta_{2k} = \frac{(-1)^k}{\pi^{2k}}\zeta_{\mathcal P}(\brk{2}^k).$$
The corollary is immediate by comparing the two expressions for $\beta_{2k}$ above. 
\end{proof}

\begin{proof}[Proof of Corollary \ref{1.8}]
Much as in the proof of Corollary \ref{1.7} above, we have from \eqref{3} that 
$$\frac{z}{\sin z} = \sum_{k=0}^{\infty} \frac{z^{2k}}{\pi^{2k}}\sum_{\ell(\lambda) = k} \frac{1}{n_{\lambda}^2} = \sum_{k=0}^{\infty} \alpha_{2k}z^{2k}$$
with 
\begin{equation}\label{25}
\alpha_{2k} = \frac{4(2^{2k-1} - 1)\zeta(2k)}{(2\pi)^{2k}} = (-1)^k\beta_{2k}.
\end{equation}
Using the Cauchy product 
\begin{equation}\label{26}
\pwr{\sum_{k=0}^{\infty} a_kz^k}\pwr{\sum_{k=0}^{\infty} b_kz^k} = \sum_{k=0}^{\infty}z^k \sum_{n=0}^k a_nb_{k-n},
\end{equation}
we see after some arithmetic 
$$\frac{z^2}{\sin z \sinh z} = \pwr{\sum_{k=0}^{\infty} \alpha_{2k}z^{2k}}\pwr{\sum_{k=0}^{\infty} \beta_{2k}z^{2k}} = \sum_{k=0}^{\infty} \gamma_{4k}z^{4k},$$
where 
$$\gamma_{4k} = \sum_{n=0}^{2k}\alpha_{2n}\beta_{4k - 2n},$$
with $\alpha_*,\beta_*$ as in \eqref{25},\eqref{26} respectively. On the other hand, the proof of Corollary \ref{1.5} implies 
$$\frac{z^2}{\sin z \sinh z} = \frac{1}{\prod_{n=1}^{\infty} \pwr{1 - \frac{z^4}{\pi^4n^4}}} = \sum_{k=0}^{\infty} \frac{z^{4k}}{\pi^{4k}}\sum_{\ell(\lambda) = k} \frac{1}{n_{\lambda}^4},$$
thus 
$$\gamma_{4k} = \frac{1}{\pi^{4k}} \zeta_{\mathcal P}(\brk{4}^k).$$
Comparing the two expressions for $\gamma_{4k}$ above, the theorem follows, just as in the previous proof.

We can carry this approach further to find $\zeta_{\mathcal P}(\brk{2^t}^k)$ for $t>2$, much as in the proof of Corollary \ref{1.5}. For instance, to find $\zeta_{\mathcal P}(\brk{8}^k)$ we begin by noting 
\begin{small}
\begin{align*}
\pwr{\sum_{k=0}^{\infty} \frac{z^{4k}}{\pi^{4k}} \zeta_{\mathcal P}(\brk{4}^k)}\pwr{\sum_{k=0}^{\infty} \frac{(-1)^kz^{4k}}{\pi^{4k}} \zeta_{\mathcal P}(\brk{4}^k)} & = \frac{1}{\prod_{n=1}^{\infty} \pwr{1 - \frac{z^4}{\pi^4n^4}}\pwr{1+\frac{z^4}{\pi^4n^4}}}\\
& = \sum_{k=0}^{\infty} \frac{z^{8k}}{\pi^{8k}} \zeta_{\mathcal P}(\brk{8}^k).
\end{align*}
\end{small}
We compare the coefficients on the left-and right-hand sides, using \eqref{26} to compute the coefficients on the left. Likewise, for $\zeta_{\mathcal P}(\brk{16}^k)$ we compare the coefficients on both sides of the equation
$$\pwr{\sum_{k=0}^{\infty} \frac{z^{8k}}{\pi^{8k}} \zeta_{\mathcal P}(\brk{8}^k)}\pwr{\sum_{k=0}^{\infty} \frac{(-1)^kz^{8k}}{\pi^{8k}} \zeta_{\mathcal P}(\brk{8}^k)}  = \sum_{k=0}^{\infty} \frac{z^{16k}}{\pi^{16k}} \zeta_{\mathcal P}(\brk{16}^k),$$
and so on, recursively, to find $\zeta_{\mathcal P}(\brk{2^t}^k)$ as $t$ increases. It is clear from induction that $\zeta_{\mathcal P}(\brk{2^t}^k)$ is of the form ``$\pi^{2^t} \times \text{rational}$'' for all $t \in \mathbb Z^+$.
\end{proof}

\begin{proof}[Proof of Corollary \ref{1.9}]
We have already seen these principles at work in the proofs of Corollaries \ref{1.5} and \ref{1.8}. We have 
$$\pwr{\sum_{\lambda\in \mathcal P_X} \frac{z^{\ell(\lambda)}}{n_{\lambda}^s}} \pwr{\sum_{\lambda \in \mathcal P_X} \frac{(-1)^{\ell(\lambda)z^{\ell(\lambda)}}}{n_{\lambda}^s}} = \frac{1}{\prod_{n\in X} \pwr{1 - \frac{z}{n^s}} \pwr{1 + \frac{z}{n^s}}} = \sum_{\lambda \in \mathcal P_X} \frac{z^{2\ell(\lambda)}}{n_{\lambda}^{2s}}.$$
Letting $z = 1$ gives \eqref{13}. If we replace $z$ with $z^s$ we may rewrite the above equation in the form
$$\pwr{\sum_{k=0}^{\infty} z^{sk}\zeta_{\mathcal P_X}(\brk{s}^k)} \pwr{\sum_{k=0}^{\infty}(-1)^k z^{sk}\zeta_{\mathcal P_X}(\brk{s}^k)} = \sum_{k=0}^{\infty} z^{2sk} \zeta_{\mathcal P_X}(\brk{2s}^k).$$
Using \eqref{26} on the left and comparing coefficients on both sides gives the $n = 0$ case of \eqref{14}; the general formula follows from the $n=0$ case by induction.
\end{proof}

\begin{proof}[Proof of comments following Corollary \ref{1.9}]
Taking $X = \mathbb P$ we see $(-1)^{\ell(\lambda)}$ specializes to Liouville's function $\lambda (n_{\lambda}) = (-1)^{\Omega (n_{\lambda})}$ (here we are using $``\lambda"$ in two different ways), where $\Omega(N)$ is the number of prime factors of $N$ with multiplicity. That \eqref{13} therefore becomes $\zeta(s)\sum_{n=1}^{\infty} \lambda(n)/n^s = \zeta(2s)$ follows from arguments similar to the proof of Corollary \ref{1.3}.
\end{proof}

\begin{proof}[Proof of Corollary \ref{1.10}]
This identity follows immediately by taking $\mathcal P_X = \mathcal P$, $n = 0$, $k=2$ in \eqref{14} and simplifying.
\end{proof}

\begin{proof}[Proof of Theorem \ref{1.11}]
The proof of \eqref{15} is similar to Euler's proof that the number of partitions of $n$ into distinct parts is equal to the number of partitions into odd parts \cite{berndt2006number}. We expand the product
$$\varphi_{\infty}(f;q) = (1-f(1)q)(1-f(2)q^2)(1-f(3)q^3)\cdots ,$$
which results in \eqref{15}.

Identities \eqref{16} and \eqref{17} follow directly from the proof of \eqref{5},\eqref{6} above. Moreover, the proof of \eqref{18} is much like the proof of \eqref{8}. We note that
\begin{align*}
\frac{1}{\varphi_{\infty}(f;q)} &= 1 - \varphi_{\infty}(f;q)\begin{array}{c}
\sum_{(5)}
\end{array} \\
&= 1 - \frac{\sum_{(5)}}{1/\varphi_{\infty}(f;q)},
\end{align*}
and replace the term $1/\varphi_{\infty}(f;q)$ in the denominator on the right with the continued fraction in \eqref{8}.
\end{proof}

\begin{proof}[Proof of Corollary \ref{etaquotient}]
The formula follows easily from the leading identities in Theorems \ref{1.1} and \ref{1.11}. We note that 
\begin{align*}
\prod_{j=1}^{n}\prod_{k_j\in X_j}\left(1\pm f_j(k_j)q^{k_j}\right)^{\pm 1}
& =\prod_{j=1}^{n}\left(\sum_{\lambda\in\mathcal P_{X_j}^{\pm}}q^{|\lambda|}\prod_{\lambda_i\in\lambda}f_j(\lambda_i)\right)
\\
& =\prod_{j=1}^{n}\left(\sum_{k_j=0}^{\infty}q^{k_j}\sum_{\substack{\lambda\vdash k_j\\ \lambda\in\mathcal P_{X_j}^{\pm}}}\prod_{\lambda_i\in\lambda}f_j(\lambda_i)\right)
\end{align*}
and repeatedly apply Equation \ref{26} on the right. 
\end{proof}

\begin{proof}[Proof of Corollary \ref{1.12}]
The identity is immediate from Theorem \ref{1.11} by letting $z=1$ in 
$$\frac{\sinh (\pi z)}{\pi z} = \prod_{n=1}^{\infty} \pwr{1 + \frac{z^2}{n^2}} = \sum_{\lambda \in \mathcal P^*} \frac{z^{\ell(\lambda)}}{n_{\lambda}^2}.$$
\end{proof}

\begin{proof}[Proof of Corollary \ref{1.13}]
This proof proceeds much like the proofs of Corollaries \ref{1.5}, \ref{1.7}, \ref{1.8} above, only more easily. We have from \eqref{3} and Theorem \ref{1.11}, together with the Maclaurin expansion of the sine function, that 
$$\frac{\sin z}{z} = \sum_{k=0}^{\infty} \frac{(-1)^kz^{2k}}{\pi^{2k}}\zeta(\brk{2}^k) = \sum_{k=0}^{\infty} \frac{(-1)^kz^{2k}}{(2k+1)!}.$$
Comparing the coefficients of the two summations above gives $\zeta(\brk{2}^k)$. We carry this approach further to find $\zeta(\brk{2^t}^k)$ for $t > 1$. We proceed inductively from the case above. Take the identity 
$$\pwr{\sum_{k=0}^{\infty} \frac{z^{2^{t-1}k}}{\pi^{2^{t-1}k}}\zeta\left(\brk{2^{t-1}}^k\right)} \pwr{\sum_{k=0}^{\infty} \frac{(-1)^kz^{2^{t-1}k}}{\pi^{2^{t-1}k}}\zeta\left(\brk{2^{t-1}}^k\right)} = \sum_{k=0}^{\infty}\frac{z^{2^tk}}{\pi^{2^tk}} \zeta(\brk{2^t}^k)$$
and compare coefficients on the left- and right-hand sides, using \eqref{26} to compute the coefficients on the left; expressions such as the remaining ones in the statement of the corollary result. It is clear from induction that $\zeta(\brk{2^t}^k)$ always has the form $``\pi^{2^tk} \times \text{finite sum of fractions}"$.
\end{proof}

\begin{proof}[Proof of Corollary \ref{1.14}]
This proof is nearly identical to the proof of Corollary \ref{1.9}. From the associated product representations it is clear that 
$$\pwr{\sum_{\lambda \in \mathcal P_X^*} \frac{z^{\ell(\lambda)}}{n_{\lambda}^s}}\pwr{\sum_{\lambda \in \mathcal P_X^*} \frac{(-1)^{\ell(\lambda)}z^{\ell(\lambda)}}{n_{\lambda}^s}} = \sum_{\lambda \in \mathcal P_X^*} \frac{(-1)^{\ell(\lambda)} z^{2\ell(\lambda)}}{n_{\lambda}^{2s}}.$$
Letting $z = 1$ gives \eqref{24}. If we replace $z$ with $z^s$ we may rewrite the above equation as 
$$\pwr{\sum_{k=0}^{\infty} z^{sk}\zeta_{\mathcal P_X^*}(\brk{s}^k)} \pwr{\sum_{k=0}^{\infty} (-1)^kz^{sk}\zeta_{\mathcal P_X^*}(\brk{s}^k)} = \sum_{k=0}^{\infty} (-1)^kz^{2sk} \zeta_{\mathcal P_X^*}(\brk{2s}^k).$$
Again using \eqref{26} on the left and comparing coefficients on both sides gives the $n=0$ case of \eqref{22}; the general formula follows by induction.
\end{proof}

\begin{proof}[Proof of comments following Corollary \ref{1.14}]
Taking $X = \mathbb P$ in Theorem \ref{1.11} and noting that $\lambda \in \mathcal P_{\mathbb P}^*$ implies $n_{\lambda}$ is squarefree, we see $(-1)^{(\lambda)} = \mu(n_{\lambda})$, where $\mu$ denotes the classical M\"obius function; therefore, we have the identity
$$\sum_{n=1}^{\infty}\frac{\mu(n)}{n^s} = \sum_{\lambda \in \mathcal P_{\mathbb P}^*} \frac{\mu(n_{\lambda})}{n_{\lambda}^s} = \eta_{\mathcal P_{\mathbb P}^*}(s) = \frac{1}{\zeta_{\mathcal P_{\mathbb P}(s)}} = \frac{1}{\zeta(s)}.$$
On the other hand, we have $\zeta_{\mathcal P_{\mathbb P}^*}(s) = \sum_{n \text{ squarefree}} 1/n^s = \sum_{n=1}^{\infty} |\mu(n)|/n^s$.
\end{proof}
\subsection*{Acknowledgments}
I wish to thank the following colleagues of mine for their help in the writing of this paper: Jackson Morrow, for extensive assistance in typesetting and editing; Larry Rolen, for offering editorial comments and references; and Andrew Sills, for historical background related to MacMahon's work. I am also grateful to the anonymous referees, whose suggestions strengthened the piece. Moreover, I would like to express gratitude to my Ph.D. advisor, Ken Ono, for his interest and insight during the course of this study. 


\bibliographystyle{bmc-mathphys} 
\bibliography{Partition_zeta_functions.bib}       

\def\polhk#1{\setbox0=\hbox{#1}{\ooalign{\hidewidth
  \lower1.5ex\hbox{`}\hidewidth\crcr\unhbox0}}}

\begin{thebibliography}{15}
\ifx \bisbn   \undefined \def \bisbn  #1{ISBN #1}\fi
\ifx \binits  \undefined \def \binits#1{#1}\fi
\ifx \bauthor  \undefined \def \bauthor#1{#1}\fi
\ifx \batitle  \undefined \def \batitle#1{#1}\fi
\ifx \bjtitle  \undefined \def \bjtitle#1{#1}\fi
\ifx \bvolume  \undefined \def \bvolume#1{\textbf{#1}}\fi
\ifx \byear  \undefined \def \byear#1{#1}\fi
\ifx \bissue  \undefined \def \bissue#1{#1}\fi
\ifx \bfpage  \undefined \def \bfpage#1{#1}\fi
\ifx \blpage  \undefined \def \blpage #1{#1}\fi
\ifx \burl  \undefined \def \burl#1{\textsf{#1}}\fi
\ifx \doiurl  \undefined \def \doiurl#1{\textsf{#1}}\fi
\ifx \betal  \undefined \def \betal{\textit{et al.}}\fi
\ifx \binstitute  \undefined \def \binstitute#1{#1}\fi
\ifx \binstitutionaled  \undefined \def \binstitutionaled#1{#1}\fi
\ifx \bctitle  \undefined \def \bctitle#1{#1}\fi
\ifx \beditor  \undefined \def \beditor#1{#1}\fi
\ifx \bpublisher  \undefined \def \bpublisher#1{#1}\fi
\ifx \bbtitle  \undefined \def \bbtitle#1{#1}\fi
\ifx \bedition  \undefined \def \bedition#1{#1}\fi
\ifx \bseriesno  \undefined \def \bseriesno#1{#1}\fi
\ifx \blocation  \undefined \def \blocation#1{#1}\fi
\ifx \bsertitle  \undefined \def \bsertitle#1{#1}\fi
\ifx \bsnm \undefined \def \bsnm#1{#1}\fi
\ifx \bsuffix \undefined \def \bsuffix#1{#1}\fi
\ifx \bparticle \undefined \def \bparticle#1{#1}\fi
\ifx \barticle \undefined \def \barticle#1{#1}\fi
\ifx \bconfdate \undefined \def \bconfdate #1{#1}\fi
\ifx \botherref \undefined \def \botherref #1{#1}\fi
\ifx \url \undefined \def \url#1{\textsf{#1}}\fi
\ifx \bchapter \undefined \def \bchapter#1{#1}\fi
\ifx \bbook \undefined \def \bbook#1{#1}\fi
\ifx \bcomment \undefined \def \bcomment#1{#1}\fi
\ifx \oauthor \undefined \def \oauthor#1{#1}\fi
\ifx \citeauthoryear \undefined \def \citeauthoryear#1{#1}\fi
\ifx \endbibitem  \undefined \def \endbibitem {}\fi
\ifx \bconflocation  \undefined \def \bconflocation#1{#1}\fi
\ifx \arxivurl  \undefined \def \arxivurl#1{\textsf{#1}}\fi
\csname PreBibitemsHook\endcsname

\bibitem{andrews1998theory}
\begin{bbook}
\bauthor{\bsnm{Andrews}, \binits{G.E.}}:
\bbtitle{The Theory of Partitions}.
\bsertitle{Cambridge Mathematical Library}.
\bpublisher{Cambridge University Press, Cambridge}, \blocation{???}
(\byear{1998}).
\bcomment{Reprint of the 1976 original}
\end{bbook}
\endbibitem

\bibitem{dunham1999euler}
\begin{bbook}
\bauthor{\bsnm{Dunham}, \binits{W.}}:
\bbtitle{Euler: the Master of Us All}.
\bsertitle{The Dolciani Mathematical Expositions},
vol. \bseriesno{22}.
\bpublisher{Mathematical Association of America, Washington, DC},
  \blocation{???}
(\byear{1999})
\end{bbook}
\endbibitem

\bibitem{macmahon1984combinatory}
\begin{bbook}
\bauthor{\bsnm{MacMahon}, \binits{P.A.}}:
\bbtitle{Combinatory Analysis}.
\bsertitle{Two volumes (bound as one)}.
\bpublisher{Chelsea Publishing Co., New York}, \blocation{???}
(\byear{1960})
\end{bbook}
\endbibitem

\bibitem{fine1988basic}
\begin{bbook}
\bauthor{\bsnm{Fine}, \binits{N.J.}}:
\bbtitle{Basic Hypergeometric Series and Applications}.
\bsertitle{Mathematical Surveys and Monographs},
vol. \bseriesno{27}.
\bpublisher{American Mathematical Society, Providence, RI}, \blocation{???}
(\byear{1988}).
doi:\doiurl{10.1090/surv/027}.
\bcomment{With a foreword by George E. Andrews}.
\burl{http://dx.doi.org/10.1090/surv/027}
\end{bbook}
\endbibitem

\bibitem{alladi1997partition}
\begin{barticle}
\bauthor{\bsnm{Alladi}, \binits{K.}}:
\batitle{Partition identities involving gaps and weights}.
\bjtitle{Trans. Amer. Math. Soc.}
\bvolume{349}(\bissue{12}),
\bfpage{5001}--\blpage{5019}
(\byear{1997}).
doi:\doiurl{10.1090/S0002-9947-97-01831-X}
\end{barticle}
\endbibitem

\bibitem{bloch2000character}
\begin{barticle}
\bauthor{\bsnm{Bloch}, \binits{S.}},
\bauthor{\bsnm{Okounkov}, \binits{A.}}:
\batitle{The character of the infinite wedge representation}.
\bjtitle{Adv. Math.}
\bvolume{149}(\bissue{1}),
\bfpage{1}--\blpage{60}
(\byear{2000}).
doi:\doiurl{10.1006/aima.1999.1845}
\end{barticle}
\endbibitem

\bibitem{zagier2015partitions}
\begin{botherref}
\oauthor{\bsnm{Zagier}, \binits{D.}}:
Partitions, quasimodular forms, and the {B}loch--{O}kounkov theorem.
The Ramanujan Journal,
1--24
(2015).
doi:\doiurl{10.1007/s11139-015-9730-8}
\end{botherref}
\endbibitem

\bibitem{griffin2014framework}
\begin{botherref}
\oauthor{\bsnm{Griffin}, \binits{M.J.}},
\oauthor{\bsnm{Ono}, \binits{K.}},
\oauthor{\bsnm{Warnaar}, \binits{S.O.}}:
A framework of {R}ogers-{R}amanujan identities and their arithmetic properties.
Duke Math. J. (Accepted for publication), arXiv preprint arXiv:1401.7718
(2014)
\end{botherref}
\endbibitem

\bibitem{ramanujan2000collected}
\begin{bbook}
\bauthor{\bsnm{Ramanujan}, \binits{S.}}:
\bbtitle{Collected Papers of {S}rinivasa {R}amanujan}.
\bpublisher{AMS Chelsea Publishing, Providence, RI}, \blocation{???}
(\byear{2000}).
\bcomment{Edited by G. H. Hardy, P. V. Seshu Aiyar and B. M. Wilson, Third
  printing of the 1927 original, With a new preface and commentary by Bruce C.
  Berndt}
\end{bbook}
\endbibitem

\bibitem{hoffman1992multiple}
\begin{botherref}
\oauthor{\bsnm{Hoffman}, \binits{M.E.}}:
Multiple harmonic series.
Pacific J. Math.
\textbf{152}(2)
(1992)
\end{botherref}
\endbibitem

\bibitem{chamberland2013gamma}
\begin{barticle}
\bauthor{\bsnm{Chamberland}, \binits{M.}},
\bauthor{\bsnm{Straub}, \binits{A.}}:
\batitle{On gamma quotients and infinite products}.
\bjtitle{Adv. in Appl. Math.}
\bvolume{51}(\bissue{5}),
\bfpage{546}--\blpage{562}
(\byear{2013}).
doi:\doiurl{10.1016/j.aam.2013.07.003}
\end{barticle}
\endbibitem

\bibitem{besser416double}
\begin{bchapter}
\bauthor{\bsnm{Besser}, \binits{A.}},
\bauthor{\bsnm{Furusho}, \binits{H.}}:
\bctitle{The double shuffle relations for {$p$}-adic multiple zeta values}.
In: \bbtitle{Primes and Knots}.
\bsertitle{Contemp. Math.},
vol. \bseriesno{416},
pp. \bfpage{9}--\blpage{29}.
\bpublisher{Amer. Math. Soc., Providence, RI}, \blocation{???}
(\byear{2006}).
doi:\doiurl{10.1090/conm/416/07884}.
\burl{http://dx.doi.org/10.1090/conm/416/07884}
\end{bchapter}
\endbibitem

\bibitem{rolen2013strange}
\begin{barticle}
\bauthor{\bsnm{Rolen}, \binits{L.}},
\bauthor{\bsnm{Schneider}, \binits{R.P.}}:
\batitle{A ``strange'' vector-valued quantum modular form}.
\bjtitle{Arch. Math. (Basel)}
\bvolume{101}(\bissue{1}),
\bfpage{43}--\blpage{52}
(\byear{2013}).
doi:\doiurl{10.1007/s00013-013-0529-9}
\end{barticle}
\endbibitem

\bibitem{alladi1977additive}
\begin{barticle}
\bauthor{\bsnm{Alladi}, \binits{K.}},
\bauthor{\bsnm{Erd{\H{o}}s}, \binits{P.}}:
\batitle{On an additive arithmetic function}.
\bjtitle{Pacific J. Math.}
\bvolume{71}(\bissue{2}),
\bfpage{275}--\blpage{294}
(\byear{1977})
\end{barticle}
\endbibitem

\bibitem{berndt2006number}
\begin{bbook}
\bauthor{\bsnm{Berndt}, \binits{B.C.}}:
\bbtitle{Number Theory in the Spirit of {R}amanujan}.
\bsertitle{Student Mathematical Library},
vol. \bseriesno{34}.
\bpublisher{American Mathematical Society, Providence, RI}, \blocation{???}
(\byear{2006}).
doi:\doiurl{10.1090/stml/034}.
\burl{http://dx.doi.org/10.1090/stml/034}
\end{bbook}
\endbibitem

\end{thebibliography}

\newcommand{\BMCxmlcomment}[1]{}

\BMCxmlcomment{

<refgrp>

<bibl id="B1">
  <title><p>The theory of partitions</p></title>
  <aug>
    <au><snm>Andrews</snm><fnm>GE</fnm></au>
  </aug>
  <publisher>Cambridge University Press, Cambridge</publisher>
  <series><title><p>Cambridge Mathematical Library</p></title></series>
  <pubdate>1998</pubdate>
  <note>Reprint of the 1976 original</note>
</bibl>

<bibl id="B2">
  <title><p>Euler: the master of us all</p></title>
  <aug>
    <au><snm>Dunham</snm><fnm>W</fnm></au>
  </aug>
  <publisher>Mathematical Association of America, Washington, DC</publisher>
  <series><title><p>The Dolciani Mathematical Expositions</p></title></series>
  <pubdate>1999</pubdate>
  <volume>22</volume>
</bibl>

<bibl id="B3">
  <title><p>Combinatory analysis</p></title>
  <aug>
    <au><snm>MacMahon</snm><fnm>PA</fnm></au>
  </aug>
  <publisher>Chelsea Publishing Co., New York</publisher>
  <series><title><p>Two volumes (bound as one)</p></title></series>
  <pubdate>1960</pubdate>
</bibl>

<bibl id="B4">
  <title><p>Basic hypergeometric series and applications</p></title>
  <aug>
    <au><snm>Fine</snm><fnm>NJ</fnm></au>
  </aug>
  <publisher>American Mathematical Society, Providence, RI</publisher>
  <series><title><p>Mathematical Surveys and Monographs</p></title></series>
  <pubdate>1988</pubdate>
  <volume>27</volume>
  <url>http://dx.doi.org/10.1090/surv/027</url>
  <note>With a foreword by George E. Andrews</note>
</bibl>

<bibl id="B5">
  <title><p>Partition identities involving gaps and weights</p></title>
  <aug>
    <au><snm>Alladi</snm><fnm>K</fnm></au>
  </aug>
  <source>Trans. Amer. Math. Soc.</source>
  <pubdate>1997</pubdate>
  <volume>349</volume>
  <issue>12</issue>
  <fpage>5001</fpage>
  <lpage>-5019</lpage>
  <url>http://dx.doi.org/10.1090/S0002-9947-97-01831-X</url>
</bibl>

<bibl id="B6">
  <title><p>The character of the infinite wedge representation</p></title>
  <aug>
    <au><snm>Bloch</snm><fnm>S</fnm></au>
    <au><snm>Okounkov</snm><fnm>A</fnm></au>
  </aug>
  <source>Adv. Math.</source>
  <pubdate>2000</pubdate>
  <volume>149</volume>
  <issue>1</issue>
  <fpage>1</fpage>
  <lpage>-60</lpage>
  <url>http://dx.doi.org/10.1006/aima.1999.1845</url>
</bibl>

<bibl id="B7">
  <title><p>Partitions, quasimodular forms, and the {B}loch--{O}kounkov
  theorem</p></title>
  <aug>
    <au><snm>Zagier</snm><fnm>D</fnm></au>
  </aug>
  <source>The Ramanujan Journal</source>
  <publisher>Springer</publisher>
  <pubdate>2015</pubdate>
  <fpage>1</fpage>
  <lpage>-24</lpage>
  <url>http://dx.doi.org/10.1007/s11139-015-9730-8</url>
</bibl>

<bibl id="B8">
  <title><p>A framework of {R}ogers-{R}amanujan identities and their arithmetic
  properties</p></title>
  <aug>
    <au><snm>Griffin</snm><fnm>MJ</fnm></au>
    <au><snm>Ono</snm><fnm>K</fnm></au>
    <au><snm>Warnaar</snm><fnm>SO</fnm></au>
  </aug>
  <source>Duke Math. J. (Accepted for publication), arXiv preprint
  arXiv:1401.7718</source>
  <pubdate>2014</pubdate>
</bibl>

<bibl id="B9">
  <title><p>Collected papers of {S}rinivasa {R}amanujan</p></title>
  <aug>
    <au><snm>Ramanujan</snm><fnm>S</fnm></au>
  </aug>
  <publisher>AMS Chelsea Publishing, Providence, RI</publisher>
  <pubdate>2000</pubdate>
  <note>Edited by G. H. Hardy, P. V. Seshu Aiyar and B. M. Wilson, Third
  printing of the 1927 original, With a new preface and commentary by Bruce C.
  Berndt</note>
</bibl>

<bibl id="B10">
  <title><p>Multiple harmonic series</p></title>
  <aug>
    <au><snm>Hoffman</snm><fnm>ME</fnm></au>
  </aug>
  <source>Pacific J. Math.</source>
  <pubdate>1992</pubdate>
  <volume>152</volume>
  <issue>2</issue>
  <url>http://projecteuclid.org/euclid.pjm/1102636166</url>
</bibl>

<bibl id="B11">
  <title><p>On gamma quotients and infinite products</p></title>
  <aug>
    <au><snm>Chamberland</snm><fnm>M</fnm></au>
    <au><snm>Straub</snm><fnm>A</fnm></au>
  </aug>
  <source>Adv. in Appl. Math.</source>
  <pubdate>2013</pubdate>
  <volume>51</volume>
  <issue>5</issue>
  <fpage>546</fpage>
  <lpage>-562</lpage>
  <url>http://dx.doi.org/10.1016/j.aam.2013.07.003</url>
</bibl>

<bibl id="B12">
  <title><p>The double shuffle relations for {$p$}-adic multiple zeta
  values</p></title>
  <aug>
    <au><snm>Besser</snm><fnm>A</fnm></au>
    <au><snm>Furusho</snm><fnm>H</fnm></au>
  </aug>
  <source>Primes and knots</source>
  <publisher>Amer. Math. Soc., Providence, RI</publisher>
  <series><title><p>Contemp. Math.</p></title></series>
  <pubdate>2006</pubdate>
  <volume>416</volume>
  <fpage>9</fpage>
  <lpage>-29</lpage>
  <url>http://dx.doi.org/10.1090/conm/416/07884</url>
</bibl>

<bibl id="B13">
  <title><p>A ``strange'' vector-valued quantum modular form</p></title>
  <aug>
    <au><snm>Rolen</snm><fnm>L</fnm></au>
    <au><snm>Schneider</snm><fnm>RP</fnm></au>
  </aug>
  <source>Arch. Math. (Basel)</source>
  <pubdate>2013</pubdate>
  <volume>101</volume>
  <issue>1</issue>
  <fpage>43</fpage>
  <lpage>-52</lpage>
  <url>http://dx.doi.org/10.1007/s00013-013-0529-9</url>
</bibl>

<bibl id="B14">
  <title><p>On an additive arithmetic function</p></title>
  <aug>
    <au><snm>Alladi</snm><fnm>K.</fnm></au>
    <au><snm>Erd{\H{o}}s</snm><fnm>P.</fnm></au>
  </aug>
  <source>Pacific J. Math.</source>
  <pubdate>1977</pubdate>
  <volume>71</volume>
  <issue>2</issue>
  <fpage>275</fpage>
  <lpage>-294</lpage>
</bibl>

<bibl id="B15">
  <title><p>Number theory in the spirit of {R}amanujan</p></title>
  <aug>
    <au><snm>Berndt</snm><fnm>BC</fnm></au>
  </aug>
  <publisher>American Mathematical Society, Providence, RI</publisher>
  <series><title><p>Student Mathematical Library</p></title></series>
  <pubdate>2006</pubdate>
  <volume>34</volume>
  <url>http://dx.doi.org/10.1090/stml/034</url>
</bibl>

</refgrp>
} 



\end{document}